\documentclass{article}

\usepackage{amsmath}
\usepackage{amssymb}
\usepackage{amsthm}

\newtheorem{theorem}{Theorem}[section]
\newtheorem{rem}[theorem]{Remark}
\newenvironment{remark}{\begin{rem}\rm}{\end{rem}}

\newtheorem{proposition}[theorem]{Proposition}
\newtheorem{prop}[theorem]{Proposition}
\newtheorem{lemma}[theorem]{Lemma}
\newtheorem{corollary}[theorem]{Corollary}
\newtheorem{eg}[theorem]{Example}

\newtheorem{definition}[theorem]{Definition}
\newtheorem{conj}[theorem]{Conjecture}

\newtheorem{prob}[theorem]{Problem}

\newcommand{\R}{\mathbb{R}}
\newcommand{\C}{\mathbb{C}}
\newcommand{\Z}{\mathbb{Z}}
\newcommand{\E}{\mathcal{E}}
\newcommand{\dirac}{\mathsf{D}_b\!\!\!\!\!\!/\,\,\,}
\newcommand{\dob}{\mathsf{D}\!\!\!\!/}

\DeclareMathOperator{\supp}{Supp}
\DeclareMathOperator{\Lie}{Lie}
\DeclareMathOperator{\End}{End}
\DeclareMathOperator{\Ch}{Ch}
\DeclareMathOperator{\Str}{Str}
\DeclareMathOperator{\Id}{Id}
\DeclareMathOperator{\Td}{Td}
\DeclareMathOperator{\Tr}{Tr}
\DeclareMathOperator{\Th}{Th}

\DeclareMathOperator{\ind}{index}

\begin{document}
\title{An equivariant index formula in contact geometry}
\author{Sean Fitzpatrick
\\
Department of Mathematics\\University of Toronto}
\maketitle

\begin{abstract}
\noindent Given an elliptic action of a compact Lie group $G$ on a 
co-oriented contact 
manifold $(M,E)$ one obtains two naturally associated objects: 
A $G$-transversally elliptic operator $\dirac$, and an 
equivariant 
differential form with generalized coefficients $\mathcal{J}(E,X)$ defined 
in terms of a choice of contact form on $M$.

We explain how the form $\mathcal{J}(E,X)$ is natural with respect to the 
contact structure, and give a formula for the equivariant 
index of $\dirac$ involving $\mathcal{J}(E,X)$. 
A key tool is the Chern character with compact 
support developed by Paradan-Vergne \cite{PV1,PV2}.
\end{abstract}

\section{Introduction} Let $(M,E)$ be a compact contact manifold.  Suppose that the contact distribution $E$ is {\em co-oriented}, and let 
$\alpha\in\Omega^1(M)$ be a contact form 
compatible with the co-orientation.  The distribution $E$ is contact if 
and only if the restriction of $d\alpha$ to $E$ is symplectic.

We equip $E$ with a complex structure $J$ compatible with $d\alpha$,
defining an almost-Cauchy-Riemann (CR) structure $E_{1,0}\subset 
TM\otimes\C$ on $M$ whose underlying real sub-bundle is $E$ (see \cite{BG}).
We suppose $E^{0,1}=(\overline{E_{1,0}})^*$ is equipped with a compatible Hermitian 
metric $h$ and connection $\nabla$, and note the isomorphism 
$\psi:E^*\rightarrow E^{0,1}$ given in Section \ref{symbol}, equation (\ref{psi}) 
below.

Analogous to the definitions in \cite{BGV}, 
let $C(E)$ be the bundle of Clifford 
algebras over $M$ whose fibre at $x\in M$ is the Clifford algebra of 
$E^*_x$ with respect to the Riemannian metric on $E$ obtained from $h$.  
Then $\E = \bigwedge E^{0,1}$ is a spinor module for $C(E)$, with Clifford 
multiplication given for $\nu\in E_x$ by
\begin{equation}\label{clif}
\mathbf{c}(\nu) = \iota(\psi(\nu))-\epsilon(\psi(\nu)).
\end{equation}

Using the connection $\nabla$ and the Clifford multiplication (\ref{clif}) 
we can define a Spin$^c$-Dirac-like operator $\dirac$ whose principal 
symbol is 
$\sigma_b(x,\xi) = 
-i\mathbf{c}(q(\xi))$, where $(x,\xi)\in T^*M$, and $q:T^*M\rightarrow 
E^*$ denotes projection.\footnote{If the almost-CR 
structure associated to $E$ is integrable, then $M$ is a Cauchy-Riemann 
manifold, 
and $\dirac$ is given in terms of the associated tangential Cauchy-Riemann 
operator $\overline{\partial}_b$ by $\dirac = 
\sqrt{2}(\overline{\partial}_b + \overline{\partial}^*_b)$.}

Since $(\sigma_b)^2(x,\xi) = 
||q(\xi)||^2$, the support of $\sigma_b$ is the anihilator bundle $E^0$ 
of $E$, whence $\dirac$ is not elliptic. 

Suppose now that a compact Lie group $G$ acts on $M$, such that 
the action preserves the contact distribution $E$, as well as its 
co-orientation, and choose $\alpha$, $J$, $h$ and $\nabla$ to be $G$-invariant.
We require the action to be 
{\em elliptic}, meaning that $TM$ is spanned by $E$ and the vectors 
tangent to the $G$-orbits, or equivalently, that $E^0$ is transverse to 
the space $T^*_G M$ of covectors perpendicular to the $G$-orbits.

Thus, although $\dirac$ is not elliptic, the requirement of ellipticity
on the group action gives 
$\supp(\sigma_b)\cap T^*_GM\subset M\times \{0\}$, 
which implies that $\dirac$ is a $G$-transversally elliptic operator in 
the sense of Atiyah \cite{AT}, and that the principal symbol $\sigma_b$ is 
a $G$-transversally elliptic symbol in the sense of Berline-Vergne 
\cite{BV1}.

Atiyah \cite{AT} has shown that the kernel and cokernel 
of any $G$-transversally elliptic operator $P$ will define trace-class 
representations of $G$, and that 
the principal symbol of $P$ defines an element in the equivariant 
$K$-theory $K_G(T^*_GM)$.  The $G$-equivariant index of 
$P$ is well-defined, but only as a generalized fucntion on 
$G$, given by the formula \cite{AT,BV1}:
\begin{equation}
\ind^G(P)(g) = \Tr(g,\ker P) - \Tr(g,\ker P^*).\label{ind}
\end{equation}

Berline and Vergne \cite{BV1, BV2} have given a character formula which 
gives the germ of (\ref{ind}) at $g\in G$ 
in terms of 
the integral over $T^*M(g)$ of certain equivariant differential forms, as 
follows:

For a $G$-transversally elliptic symbol $\sigma$, we have, for $g\in G$ 
and $X\in\mathfrak{g}(g)$ sufficiently small,

\begin{equation}
\ind^G(\sigma)(ge^X) = \int\limits_{T^*M(g)}(2i\pi)^{-\dim M(g)} 
\frac{\Ch^g_{BV}(\sigma, X)\hat{A}^2(M(g),X)}{D_\R(N(g),X)},\label{bv}
\end{equation}

where $\Ch_{BV}^g(\sigma, X)$ is the Chern character of \cite{BV1}.
For a $G$-transversally elliptic operator $P$ with principal symbol 
$\sigma(P)$, the equality of generalized functions $\ind^G(P) = 
\ind^G(\sigma(P))$ was proved in \cite{BV2}.

Recent work of Paradan and Vergne \cite{PV2} allows one to replace the 
non-compactly supported equivariant forms in the integrand of (\ref{bv}) 
by forms with compact support, 
provided one passes to equivariant differential forms with generalized 
coefficients: these are $C^{-\infty}$ maps from $\mathfrak{g}$ to 
$\mathcal{A}(M)$, as in \cite{KV}.  The space of all such forms will be 
denoted by $\mathcal{A}^{-\infty}(\mathfrak{g},M)$.

When one allows generalized coefficients, it is possible to define a 
natural differential form on $M$ adapted to the contact structure as 
follows:

Let $\alpha$ be a contact form on $M$, let $D = d-\iota(X)$ be 
the 
equivariant 
differential, and  let $\delta_0$ 
be the Dirac delta distribution 
on $\R$.  Then we may define the form 
\begin{equation}
\mathcal{J}(E,X) = \alpha\wedge\delta_0(D\alpha(X)),\label{form}
\end{equation} 
which is well-defined as an element of 
$\mathcal{A}^{-\infty}(\mathfrak{g},M)$.

Moreover, using the properties of the delta distribution, one has that
\begin{enumerate}
\item $D\mathcal{J}(E,X) = 0$, so that $\mathcal{J}(E,X)$ defines a class 
in $\mathcal{H}^{-\infty}(\mathfrak{g},M)$, the 
equivariant cohomology of $M$ with generalized coefficients.

\item $\mathcal{J}(E,X)$ is independent of the choice of contact form 
$\alpha$ 
and 
thus depends only on the contact structure $E$.
\end{enumerate}

For a fixed $g\in G$, let $i:M(g)\rightarrow M$ denote the inclusion of 
the set of $g$-fixed points in $M$.  In Proposition \ref{prop1}, we show 
that $(M(g), E(g))$ is again a contact manifold, with contact form 
$\alpha^g = i^*\alpha$, so that $\mathcal{J}(E(g),X) = 
\alpha^g\wedge\delta_0(D\alpha^g(X))$ is again well defined, 
for $X\in 
\mathfrak{g}(g)\subset\mathfrak{g}$.

In this article, our interest in the form $\mathcal{J}(E,X)$ is due to 
its appearance in our formula 
for the equivariant index of the $G$-transversally elliptic operator 
$\dirac$.  The results of \cite{PV1,PV2} allow us to re-write the 
integrand of (\ref{bv}) in terms of a Chern character 
$\Ch_{MQ}(\sigma,X)$ with ``gaussian shape'' along the fibres of $E^*$ 
in the sense of \cite{MQ}, and a differential form $P_\lambda(X)$ with 
generalized coefficients whose support intersects $E^0$ in a compact set. We 
are then able to integrate over the fibres of $T^*M(g)$ to obtain:

\begin{theorem}
Let $(M,E)$ be a compact, co-oriented contact manifold, 
and let $G$ be a 
compact Lie group acting elliptically on $M$.
The $G$-equivariant index of $\dirac$ is the generalized 
function on $G$ whose germ at $g\in G$ is given,
for $X\in\mathfrak{g}(g)$ sufficiently small, by
\begin{equation}
\ind^G(\dirac)(ge^X) = \int\limits_{M(g)}(2\pi i)^{-k(g)} 
\frac{\Td(E(g),X)\mathcal{J}(E(g),X)}{D_\C(N(g),X)},
\label{a}
\end{equation}
where $\dim(M(g)) = 2k(g)+1$.
\end{theorem}

In particular, we have the following formula at the identity:
\begin{theorem}
For $X\in \mathfrak{g}$ sufficiently small,
\begin{equation}
\ind^G(\dirac)(e^X) = \frac{1}{(2\pi i)^n}\int\limits_M 
\Td(E,X)\mathcal{J}(E,X).\label{b}
\end{equation} 
\end{theorem}

In the case of an elliptic circle action on $M$, we can relate our index formula to the index of a Dirac operator on the quotient: Suppose that $M$ is a pre-quantum $U(1)$-bundle over a pre-quantizable Hamiltonian $G$-space $(B,\omega,\Phi)$ in the sense of \cite{GGK}. If $\tilde{\alpha}$ is a connection form satisfying the prequantization condition $iD\tilde{\alpha}(X) = \pi^*(\omega-\Phi(X))$, then $\alpha = i\tilde{\alpha}$ is a contact form on $M$, and the action of $G\times U(1)$ on $M$ obtained from the action of $G$ on $B$ and the principal $U(1)$-action is elliptic.

\begin{corollary}
Let $\dob$ denote the Dolbeault-Dirac operator acting on sections of $\bigwedge T^{0,1}B$, let $\mathbb{L} = M\times_{U(1)}\C$, and denote $\sigma_m = \sigma(\dob)\otimes \Id_{\mathbb{L}^{\otimes m}}$. Then we have the following equality of generalized functions on $G\times U(1)$:
\begin{equation*}
\ind^{G\times U(1)}(\dirac)(g,u) = 
\sum\limits_{m\in\Z}u^{-m}\ind^G(\sigma_{m})(g).
\end{equation*}
\label{ep}
\end{corollary}

\section{Elliptic actions on contact manifolds}
Let $G$ be a compact Lie group, and let $M$ be a $G$-manifold.  We make 
use of the following notation:

\begin{definition}
The set $\bigcup_{x\in M}(T_x(G.x))^0 \subset 
T^\ast M$ of covectors orthogonal to the $G$-orbits will be denoted 
$T^*_GM$.
\end{definition}

\begin{definition}
Let $\eta\in\Omega^1 (M)$ be an invariant 1-form on a $G$-manifold $M$.  
We define 
the  $\boldsymbol{\eta}${\bf -moment map} to be the map $f_\eta: 
M\rightarrow 
\mathfrak{g}^\ast$ 
given by the pairing
\begin{equation*}
<f_\eta(m), X> = \eta_m(X_M(m)),
\end{equation*}
for any $X\in \mathfrak{g}$, where $X_M$ is the vector field on $M$ 
corresponding to $X$ via the infinitesimal action of $\mathfrak{g}$ on 
$M$.

We denote by $C_\eta$ the zero-level set $f_\eta^{-1}(0)\subset M$ of the 
$\eta$-moment map.\label{mo}
\end{definition}

For any $G$-space $V$, we will denote by $V(g)$ the 
subset of $V$ fixed by the action of an element $g\in G$.

\begin{remark}
Let $\theta$ be the canonical 1-form on $T^*M$, and consider the 
lift of the action of $G$ on $M$ to $T^*M$.  This action is Hamiltonian, 
and $f_\theta:T^*M\rightarrow \mathfrak{g}^*$ is the 
corresponding moment map.  We may describe the space $T^*_G 
M$ according to $T^*_G M = C_\theta$.
\end{remark}

Let $(M,E)$ denote a compact co-oriented contact manifold, and let $G$ 
be a compact Lie group acting on $M$ preserving the contact structure 
$E$, and the co-orientation.   
Choose a $G$-invariant contact form $\alpha$ compatible with the 
co-orientation.  That is, if we let $E^0_+$ be the connected 
component of $E^0\setminus 0$ that is positive with respect to the 
co-orientation, then $\alpha(M)\subset E^0_+$.
We will suppose such a choice of contact form has been fixed throughout 
the article.\footnote{A good exposition of this terminology can be found 
in \cite{ler}.}

\begin{remark}
Recall that the space $E^0_+\subset T^*M$ is a symplectic cone over the 
base $M$, called the {\em symplectization} of $M$.  The symplectic form on 
$E^0_+$ is the pullback under inclusion of the canonical symplectic form 
on $T^*M$.  The 
cotangent lift of an action of $G$ on $M$ preserving $E$ restricts to a 
symplectic action of $G$ 
on $E^0_+$ commuting with the natural $\R_+$ action.
\end{remark}

\begin{definition}
The action of $G$ 
on $(M,E)$ is said to be {\bf elliptic} if and only if $T^\ast_G M\cap 
E^0 = 0$.
\end{definition}
For the remainder of this article, we will impose this stronger condition 
on the action of $G$ on $M$.

\begin{remark}\label{mm}
The action of $G$ on $(M,E)$ is elliptic if 
the orbits of $G$ in $M$ are nowhere tangent to the contact distribution.  
Alternatively, if $\Phi:E^0_+\rightarrow \mathfrak{g}^*$ is the 
restriction of $f_\theta$ to $E^0_+$, then the action is elliptic if and 
only if zero is not in the image of $\Phi$.
\end{remark}
 
Associated to the chosen contact form $\alpha$ 
is the {\em Reeb vector field}, which is the vector field $Y\in 
\Gamma(TM)$ 
such that
\begin{equation*}
\iota(Y)\alpha =1 \:\:\mbox{ and } \quad \iota(Y)d\alpha = 0.
\end{equation*}

Accordingly, we obtain a splitting $TM = 
E\oplus \R Y$, dual to the splitting $T^\ast M = E^\ast \oplus \R\alpha$ 
given 
by the choice of contact form.

The following proposition is a key lemma for our proof of the fixed-point 
formula (\ref{a}):

\begin{proposition}\label{prop1}
Let $(M,\alpha)$ be a co-oriented contact manifold, and suppose $G$ is a 
compact group acting on $M$ elliptically.  For 
any $g\in G$, let $i:M(g)\rightarrow M$ denote inclusion of the $g$-fixed 
points.  Then we have:
\begin{enumerate}
\item The submanifold $M(g)\subset M$ is a
contact manifold, and $\alpha^g = i^*\alpha$ is a contact form on $M(g)$.

\item The action of the centralizer $G(g)$ of $g$ in $G$ on $M(g)$ is 
elliptic.
\end{enumerate}
\end{proposition}

\begin{proof}
(1) Let $TM = E\oplus\R Y$, where $Y$ is the Reeb 
vector field asociated to the $G$-invariant contact form $\alpha$, and 
thus $G$-invariant as well.

Denote by $N(g)$ the normal bundle to $M(g)$ in $M$.  Then we have that
\begin{equation*}
TM\vert_{M(g)} = TM(g)\oplus N(g) = E(g)\oplus\R Y^g\oplus N(g)
\end{equation*}
by the invariance of $Y$, where $Y^g = Y\vert_{M(g)}$ and $E(g)$ is the 
subset of $E$ fixed by the action of $g\in G$.

Choose any $m\in M(g)$.  Then we know 
that 
$d\alpha\vert_{E_m}$ is symplectic.  Moreover, the action of $G$ on the 
symplectic vector space $E_m$ preserves $\alpha$, and thus the symplectic 
structure, whence $E_m(g)$ is symplectic, with symplectic form 
$d\alpha^g = i^*d\alpha$.

Finally, since $\alpha(M(g))\subset T^*M(g) = (TM(g))^*$, and 
$\iota(Y^g)\alpha^g = 1$, we have $\ker(\alpha^g)\cap\R Y^g = 0$, 
whence $\ker(\alpha^g)\subset E(g)$, and a dimension count gives 
$\ker(\alpha^g) = E(g)$.

(2) Recall that the action of $G$ on $M$ is elliptic if and only if zero 
is not in the image of the moment map $\Phi:E^0_+\rightarrow 
\mathfrak{g}^*$ (Remark \ref{mm}).  Let $H=G(g)$, and let $\mathfrak{h}$ 
be the 
Lie algebra of $H$.

If $x\in E^0_+(g)$, we have by the equivariance of $\Phi$ that 
$\Phi(x)\in\mathfrak{h}^*$.  Thus, the corresponding moment map 
$\Psi:E^0_+(g)\rightarrow \mathfrak{h}^*$ for the action of $G(g)$ on 
$E^0_+(g)$ is simply the restriction of $\Phi$ to $E^0_+$ (the projection 
from $\mathfrak{g}^*$ to $\mathfrak{h}^*$ being redundant).

Since $\Psi = \Phi\vert_{E^0_+(g)}$, it follows that zero is not in the 
image of $\Psi$, and thus the action of $G(g)$ on $M(g)$ is elliptic. 
\qedhere
\end{proof}

\section{Definition of the operator $\dirac$}\label{symbol}
Let $\alpha\in\Gamma(E^0\setminus 0)$ be a given choice of $G$-invariant contact form 
on $(M,E)$.  Then $d \alpha\vert_E$ defines a symplectic structure on the 
fibres of $E$, so that $E$ is a symplectic vector bundle over $M$.  Let $J$ be a $G$-invariant complex structure on $E$.

If $\beta$ is any other contact form, then $\beta = e^f\alpha$ 
for some $f\in C^\infty(M)$.  Thus we have $d\beta = e^f d f\wedge\alpha 
+e^f d\alpha$, whence $d\beta\vert_E = e^f d\alpha\vert_E$.  Therefore, 
if $J$ is a complex structure on $E$ compatible with $d\alpha$, 
it is also compatible with $d\beta$, and thus depends only on the contact 
structure $E$. 

The pair $(E,J)$ determines an almost-CR structure 
$E_{1,0}$ on $M$ whose 
underlying real bundle is the contact distribution (see \cite{BG}).
Thus $E_{1,0}\cap\overline{E_{1,0}} = 0$, and so $E_{1,0}\oplus 
E_{0,1} = E\otimes \C$, where $E_{0,1} = \overline{E_{1,0}}$.
Let $E^{1,0} = (E_{1,0})^*$, giving the $G$-invariant decomposition
\begin{equation*}
E^\ast\otimes\C = E^{1,0}\oplus E^{0,1}
\end{equation*}
into $\pm i$-eigenspaces of $J$, where for $\eta\in T^*M$ and $\xi\in TM$ 
the induced almost-complex structure on $E^*$ is given by $J(\eta)(\xi) = 
\eta(J(\xi))$

Let $\psi:E^\ast\xrightarrow{\simeq}E^{0,1}$ be the isomorphism given by
\begin{equation}
\psi(\eta) = \eta + iJ(\eta).\label{psi}
\end{equation}
Let $p:E^*\rightarrow M$, $q:T^\ast M\rightarrow E^\ast$, and $\pi_M = p\circ q : 
T^\ast M\rightarrow M$ denote projections.
Let $h$ be the $G$-invariant Hermitian metric on $E^{0,1}$ determined by 
$J$ and $d\alpha$, and let $\nabla$ be a $G$-invariant Hermitian 
connection on $E^{0,1}$.
The metric $h$ determines an invariant Riemannian metric on $E$.  Let 
$C(E)\rightarrow M$ be the bundle whose fibre over $x\in M$ is the 
Clifford algebra of $E_x$ with respect to this metric.  For any $\nu\in 
E_x$, we have the Clifford multiplication given by
\begin{equation}
\mathbf{c}(\nu) = \iota(\psi(\nu))-\epsilon(\psi(\nu)),
\end{equation}
where $\epsilon(\eta)$ denotes exterior multiplication by $\eta$, and 
$\iota(\eta)$ denotes contraction with respect to $h$: $\iota(\eta,\xi) = 
h(\eta,\xi)$.
The multiplication $\mathbf{c}$ makes $\E=\bigwedge E^{0,1}$ into a spinor 
module for $C(E)$.

Using $\nabla$ and $\mathbf{c}$, we define a $G$-invariant 
differential operator $\dirac:\Gamma(\E)\rightarrow \Gamma(\E)$ by the composition
\begin{equation}\label{dirac}
\Gamma(\E)\xrightarrow{\nabla} \Gamma(T^*M\otimes\E) \xrightarrow{q} 
\Gamma(E^*\otimes\E)\xrightarrow{\mathbf{c}}\Gamma(\E).
\end{equation}
In the case that $E_{1,0}$ is integrable, $M$ is a Cauchy-Riemann manifold, and we have
\begin{equation*}
\dirac = \sqrt{2}(\overline{\partial}_b + \overline{\partial}^*_b),
\end{equation*}
where $\overline{\partial}_b$ is the tangential Cauchy-Riemann operator 
defined with respect to $E_{1,0}$ \cite{BG}.

Write $\E = \E^+\oplus \E^-$ with respect to the $\Z_2$-grading given by 
exterior degree, and let $\sigma_b :\pi^*_M\E^+\rightarrow \pi^*_M\E^-$ 
denote the principal symbol of $\dirac$, given for any $(x,\xi)\in T^*M$ and $\gamma\in\E^+$ by
\begin{equation}\label{sig}
\sigma_b(x,\xi)(\gamma) = -i\mathbf{c}(q(\xi))\gamma.
\end{equation}
For any morphism $\sigma$ on $\pi^*_M \E$, define 
\begin{equation*}
\supp(\sigma) = \{(x,\xi)\in T^*M|\sigma(x,\xi) \text{ is not invertible}\}.
\end{equation*}
Since 
$\sigma^2_b(x,\xi) = ||q(\xi)||^2$, we have $\supp(\sigma_b) = E^0$.
This implies that for an elliptic $G$-action on $M$, $\dirac$ 
is a $G$-transversally elliptic differential operator in the sense of 
Atiyah \cite{AT}, since $E^0\cap T^\ast_G M = 0$.

Therefore, $[\pi_M^\ast\E, \sigma_b]$ defines an equivariant 
$K$-theory class in $K_G(T^\ast_G M)$.  This class is independent of the 
almost-CR structure (since any two such structures with underlying real 
bundle $E$ are homotopic) and the 
Hermitian metric. A formula for the equivariant 
index of this class has been given by Berline and Vergne \cite{BV1, 
BV2}, but requires the integration of non-compactly supported forms on 
$T^\ast M$.

Following Paradan and Vergne \cite{PV2}, we will instead pass to 
equivariant differential forms with generalized coefficients, which will 
allow us to construct a compactly supported form whose integral over 
$T^\ast M$ agrees with that of the Berline-Vergne formula, and for which 
the integral over the fibres is easily carried out.

Before dealing with the technical details of this construction and the 
proof of our main theorem, we pause to consider two simple examples in 
which our index theorem may be applied.

\section{Examples}
The simplest examples of an elliptic group action on a contact manifold 
involve free circle actions.  A particularly simple example is discussed 
in \cite{V}: that of a circle acting on itself by multiplication.

\vspace{10pt}

{\bf Example 1: $S^1$}

Consider the circle $S^1 = \{e^{i\theta}|\theta\in\R\}$.  The form 
$d\theta$ is a contact form on $S^1$, with the zero section as the contact 
distribution.  The group $U(1) = \{e^{i\phi}\}$ acts freely on $S^1$ by 
multiplication.  The action is elliptic, since $T^*_G S^1 = 0$ (while $E^0 
= T^* S^1$).

Here, our operator is $\dirac=0$, and since $T^*_G S^1 = 0$, even the zero 
operator on $S^1$ is $U(1)$-transversally elliptic.
The $U(1)$-equivariant index is given simply by

\begin{equation*}
\ind^G(0)(e^{i\phi}) = \int\limits_{S^1}\mathcal{J}(\phi) = 
2\pi\delta_0(\phi) = \sum_{m\in\Z}e^{im\phi},
\end{equation*}
where the last equality is valid for $\phi$ sufficiently small, using the 
Poisson summation formula for $\delta_0$.

\vspace{10pt}
{\bf Example 2: $S^3$}

Let $M=S^3$ be the unit sphere in $\R^4$ 
with 
co-ordinates $(x_1, y_1, x_2, y_2)$, and consider the frame $\{X,Y,T\}$ 
for 
$TS^3$ given by

\begin{align*}
X&= x_2\frac{\partial}{\partial x_1}-y_2\frac{\partial}{\partial y_1} 
-x_1\frac{\partial}{\partial x_2}+y_1\frac{\partial}{\partial 
y_2}\\
Y&= -y_2\frac{\partial}{\partial x_1}- x_2\frac{\partial}{\partial y_1} + 
y_1\frac{\partial}{\partial x_2} +x_1\frac{\partial}{\partial y_2}\\ 
T&= y_1\frac{\partial}{\partial x_1}- x_1\frac{\partial}{\partial y_1} 
+y_2\frac{\partial}{\partial x_2}-x_2\frac{\partial}{\partial 
y_2}.
\end{align*}

A contact structure is given by $E=TS^3/\R T$.  If we let 
$\{\xi,\zeta,\alpha\}$ denote the corresponding co-frame, then 
$\alpha$ 
is a contact form on $S^3$.  In co-ordinates we have

\begin{equation*} 
\alpha = y_1dx_1-x_1dy_1+y_2dx_2-x_2dy_2,
\end{equation*}
and one readily sees that $\alpha(T) = x_1^2+y_1^2+x_2^2+y_2^2 = 1$, so 
that $T$ is the Reeb field for $\alpha$.

We let $U(1)$ act on $S^3$, with action given in complex 
co-ordinates as follows: identify 
$\R^4 \cong \C^2$ via $z_j = x_j + iy_j$, $j=1,2$.  The action of 
$e^{i\phi}\in U(1)$ on $\C^2$ by $e^{i\phi}\cdot (z_1,z_2) = 
(e^{i\phi}z_1,e^{i\phi}z_2)$ restricts to an action of $U(1)$ on 
$S^3$.  Let $\mathfrak{g} = i\R$ denote the Lie algebra of $G$, and note 
that the infinitesimal action of $\mathfrak{g}$ on $M$ is given by 
$i\phi\mapsto \phi T$. The orbits of the action are thus 
transverse to the contact distribution $E$, whence the action 
of $U(1)$ on $S^3$ is elliptic. 

The almost-CR structure on $M$ is given by taking $E_{1,0} = \C Z$, 
where $Z = \frac{1}{\sqrt{2}}(X+iY)$.
The corresponding covector in $E^*\otimes\C$ is $\theta = 
\frac{1}{\sqrt{2}}(\xi-i\zeta)$. 
The associated complex structure on $E$ comes from the 
complex structure on $\C^2$, and is given by $J(X) = -Y$ and $J(Y) 
= X$, so that $J(\xi) = \zeta$, and $J(\zeta) = -\xi$ on $E^*$.  
Since this structure is integrable, $M$ is a 
CR manifold, and $\dirac = \sqrt{2}(\overline{\partial}_b + 
\overline{\partial}^*_b)$.

Writing $\eta\in T^*_xS^3$ as $\eta = 
a\xi+b\zeta+c\alpha$, the symbol of $\dirac$ is given by
\begin{equation*}
\sigma_b(x,\eta) = 
-i\sqrt{2}((a+ib)\iota(\overline{Z}) - 
(a-ib)\epsilon(\overline{\theta})),
\end{equation*}
from which we see that $\sigma_b^2(x,\eta) = a^2 + b^2$.

Finally, for $\phi$ sufficiently small, the $U(1)$-equivariant index of $\dirac$ is given by
\begin{align}
\ind^{U(1)}(\dirac)(e^{i\phi}) &= 
\frac{1}{2\pi i}\int\limits_{S^3}\Td(E,\phi)\mathcal{J}(E,\phi) 
= 2\pi(\delta_0(\phi) - i\delta_0^\prime(\phi))\nonumber\\
&= \sum_{m\in\Z}(1-m)e^{im\phi} = \sum_{m\in\Z}\frac{1}{2\pi 
i}\int\limits_{S^2}\Td(S^2)e^{-im(\omega - \phi)},\label{rr}
\end{align}
since
\begin{equation*}
\mathcal{J}(E,\phi) = \alpha\wedge\delta_0(d\alpha - \phi) = 
\alpha(\delta_0(-\phi) + \delta_0^\prime(-\phi)d\alpha),
\end{equation*}
while, if $\pi:S^3\rightarrow S^2$ denotes the projection onto the orbit space, then
\begin{align*}
\Td(E,\phi) &= \Td(\pi^*TS^2,\phi) = \pi^*\Td(S^2)\\
	    &= \pi^*(1+i\omega) = 1+id\alpha.
\end{align*}

Note that the last equality in (\ref{rr}) is an instance of Corollary 
\ref{pb}.

\section{Equivariant differential forms with \newline generalized coefficients}

Let $N$ be a smooth manifold, not necessarily compact.  Let 
$\mathcal{A}^\infty(\mathfrak{g}, N)$ denote the complex of smooth 
equivariant differential forms on $N$; that is, the space of smooth maps 
$\alpha :\mathfrak{g} = \Lie(G) \rightarrow \mathcal{A}(N)$ that commute 
with the actions of $G$ on $\mathfrak{g}$ and $N$.

The space $\mathcal{A}^\infty(\mathfrak{g},N)$ is equipped with the 
equivariant differential $D$, given by
\begin{equation*}
(D\alpha)(X) = d(\alpha(X)) - \iota(X_N)\alpha(X),
\end{equation*}
for any $X\in\mathfrak{g}$, where $\iota(X_N)$ is contraction by the 
vector field on $N$ generated by the infinitesimal action of 
$X\in\mathfrak{g}$.  We have $D^2 = 0$ on 
$\mathcal{A}^\infty(\mathfrak{g},N)$, whence we can define the 
equivariant cohomology $\mathcal{H}^\infty(\mathfrak{g},N)$.

We can define further the complex $\mathcal{A}^{-\infty}(\mathfrak{g},N)$ 
of equivariant differential forms with generalized coefficients: this is 
the space of G-equivariant generalized functions from $\mathfrak{g}$ to 
$\mathcal{A}(N)$ \cite{KV}.  Thus, for any 
$\alpha\in\mathcal{A}^{-\infty}(\mathfrak{g},N)$ and any compactly 
supported test function $\phi\in C^\infty(\mathfrak{g})$, 
$\int_{\mathfrak{g}}\alpha(X)\phi(X)dX$ is a smooth differential form on $N$.

The equivariant differential $D$ extends to 
$\mathcal{A}^{-\infty}(\mathfrak{g},N)$, with $D^2 = 0$, and so it is 
possible to define the space 
the space $\mathcal{H}^{-\infty}(\mathfrak{g},N)$ of equivariant 
cohomology with generalized coefficients.  See \cite{KV,PV,V} for examples 
where this space is computed.

For non-compact $N$, we may consider as well the cohomology 
of (generalized) equivariant 
differential forms with compact support on $N$, denoted by 
$\mathcal{H}^{\pm\infty}_c(\mathfrak{g},N)$.

\subsection{Forms in $\mathcal{A}^{-\infty}(\mathfrak{g},N)$ derived from 
distributions on $\R$}

It will be useful in our 
computation of the equivariant index to work in terms of the following 
generalized functions on $\R$:
\begin{equation*}
\delta_+(x) = 
\frac{i}{2\pi}\lim_{\epsilon\rightarrow 
0}\frac{1}{x+i\epsilon},\quad 
\delta_-(x) = 
\frac{-i}{2\pi}\lim_{\epsilon\rightarrow 
0}\frac{1}{x-i\epsilon}.
\end{equation*}
Note that we have
\begin{equation*}
\delta_+(x) + \delta_-(x) = \frac{1}{\pi}\lim_{\epsilon\rightarrow 
0}\frac{\epsilon}{x^2 + \epsilon^2},
\end{equation*}
which we identify as the Dirac delta distribution $\delta_0(x)$, giving 
the first of the 
following identities:
\begin{equation} \delta_+ + \delta_- = \delta_0,\end{equation}\label{d1}
\begin{equation}-ix\delta_+(x) = ix\delta_-(x) = 1,\quad x\delta_0(x) =0, 
\label{d2}\end{equation} 
and for any $a\in\R\setminus\{0\}$, we have
\begin{equation}
a\delta_0(ax) = \begin{cases} \delta_0(x),& \text{if 
$a>0$}\\ -\delta_0(x),& \text{if $a<0$}\end{cases}, \quad 
a\delta_\pm(ax) = \begin{cases} \delta_\pm(x) & \text{if 
$a>0$}\\ \label{d3}
-\delta_\mp(x) & \text{if $a<0$}\end{cases}.
\end{equation}

The integral representations of these 
generalized functions given by
\begin{equation*}
\delta_+(x) = \frac{1}{2\pi}\int^\infty_0 e^{ixt} d t, \quad \delta_-(x) = 
\frac{1}{2\pi}\int^0_{-\infty} e^{ixt} d t
\end{equation*}
will be helpful in the computation of the index formulas to 
follow.

Distributions on $\R$, such as the above, can be used to define 
equivariant differential 
forms with generalized coefficients.  For example, 
given an invariant 1-form $\beta$ on $N$, the form $\delta_0(D\beta)$, 
which we may view as the oscillatory integral 
$\int^\infty_{-\infty}e^{-itD\beta} d t$, 
is well-defined as a generalized equivariant form wherever the pairing 
$X\mapsto \beta(X_N)$ is non-zero.

Let us explain in general how such a form is well-defined: Let $u\in 
C^{-\infty}(\R)$ be a 
distribution on $\R$. We may 
consider its pull-back by a smooth, proper map $h:\mathfrak{g}\rightarrow 
\R$, 
which will give 
a well-defined distribution $h^*u = u\circ h$ on $\mathfrak{g}$ provided 
\begin{equation}
h^*(WF(u))\cap (\mathfrak{g}\times\{0\})=\emptyset \subset 
T^*\mathfrak{g},\label{wf}
\end{equation}
where $WF(u)$ denotes the wavefront set of $u$ \cite{FJ}.

\begin{remark}\label{blah}
Note that for the resulting distribution on $\mathfrak{g}$, we 
have $WF(h^*u)\subset h^*WF(u)$.  Furthermore, for any derivative 
$u^{(j)}$ of $u$, we have $WF(u^{(j)})\subset 
WF(u)$, so that if a map $h$ satisfies the condition above with respect to 
$u$, it does so for all the derivatives of $u$ as well.
\end{remark}

Now, if $\beta$ is an invariant 1-form on $M$, then
for a fixed point $m\in M$, $f_\beta(m)$ gives us a linear map 
from 
$\mathfrak{g}$ to $\R$.  If $f_\beta(m)$ satisfies (\ref{wf}) 
for all $m\in M$, then we may set
\begin{equation}
u(D\beta)(X) = u(d\beta-f_\beta(X))= \sum_j 
\frac{u^{(j)}(-f_\beta(X))}{j!}(d\beta)^j,\label{defu}
\end{equation}
which is well-defined by Remark \ref{blah} 

On a contact manifold $(M,E)$ on which a Lie 
group $G$ acts elliptically, consider the form 
\begin{equation}
\mathcal{J}(E,X) = \alpha\wedge\delta_0(D\alpha(X)),\label{J}
\end{equation}
where $\alpha$ is any contact form.

The ellipticity hypothesis ensures that the pairing $X\mapsto \alpha(X_M)$ 
is non-zero:
We have 
$WF(\delta_0) = \{0\}\times (\R\setminus 0)$, while $f_\alpha^{-1}(0) = 
\emptyset$.  Thus for all $m\in M$, $\eta = f_\alpha(m)$ is non-zero, and 
\begin{equation*}
\eta^*(WF(\delta_0)) = \{(X, t\eta)\in \mathfrak{g}\times 
\mathfrak{g}^*|(\eta(X), t)\in \{0\}\times (\R\setminus 0)\}.
\end{equation*}
Since $t\eta$ is never zero, (\ref{wf}) is satisfied.

Using the properties of the delta distribution given above, we obtain the 
following: 

\begin{proposition}
Let $(M,\alpha)$ be a co-oriented contact manifold on which a Lie group 
$G$ acts elliptically. Then the form $\mathcal{J}(E,X)$ is equivariantly 
closed, 
and independent of the choice of contact form.
\end{proposition}

\begin{proof}
We have: 
\begin{equation*}
D(\alpha\wedge \delta_0(D\alpha)) = D\alpha\wedge\delta_0(D\alpha) = 0 
\:\mbox{ by (\ref{d2}) },
\end{equation*}
while if we change $\alpha$ to $e^f\alpha$ for some $f\in C^\infty(M)$ we 
have using (\ref{d3}) that
\begin{align*}
e^f\alpha\wedge \delta_0(D(e^f\alpha)) &= e^f\alpha\wedge\delta_0(e^f(d 
f\wedge \alpha + D\alpha))\\
&= \alpha\wedge\delta_0(df\wedge\alpha +D\alpha) = 
\alpha\wedge\delta_0(D\alpha),
\end{align*}
where in the last equality we have used (\ref{defu}), 
and the fact that $\alpha\wedge\alpha = 0$. 
\end{proof}

\begin{remark}
Recall \cite{GGK} that given a symplectic manifold $(N,\omega)$ of 
dimension 2n and a Hamiltonian action of a Lie group $G$ on $N$ with 
moment map $\Psi$ we may 
define the Duistermaat-Heckman distribution $u_{DH}$ on 
$C^\infty_c(\mathfrak{g}^*)$ by
\begin{equation*}
<\phi, u_{DH}> =\int\limits_{\mathfrak{g}^*}\phi u_{DH} = 
\frac{1}{(2\pi)^n}\int\limits_N (\Psi^*\phi) 
e^\omega.
\end{equation*}

The Fourier transform of $u_{DH}$ is given, for $h\in 
C^\infty_c(\mathfrak{g})$ by
\begin{equation*}
<\widehat{u_{DH}},h> = <u_{DH},\hat{h}> = 
\int\limits_{\mathfrak{g}}h(X)I(X)dX,
\end{equation*}
so that $\widehat{u_{DH}} = I(X)dX$, where
\begin{equation*}
I(X) = \int\limits_{\mathfrak{g}*}e^{-i<X,\xi>}u_{DH}(\xi) = 
\frac{1}{(2\pi)^n}\int\limits_N e^{-i<X,\Psi>}e^\omega = (2\pi 
i)^{-n}\int\limits_N 
e^{i\omega(X)}.
\end{equation*}

Now, given a co-oriented contact manifold $(M,E)$ of dimension $2n+1$, 
consider the 
anihilator $E^0$ of $E$.  Although not quite a symplectic manifold, since 
the form $\omega = d(t\alpha)$ is degenerate for $t=0$, we have 
the moment map $\Psi = tf_\alpha$, and if we compute $I(X)$ in this case, 
we find
\begin{equation*}
I(X) = \frac{1}{(2\pi i)^{n+1}}\int\limits_{E^0} e^{i\omega(X)}
=\frac{1}{(2\pi i)^n}\int\limits_M\mathcal{J}(E,X).
\end{equation*}
Similarly, on the symplectic manifold $E^0_+$ we obtain an expression for the Fourier transform of the Duistermaat-Heckman distribution by replacing $\delta_0$ by $\delta_+$.
\end{remark}

\subsection{Cohomology with support}
We give here a quick summary of the material in \cite{PV1} and \cite{PV2} 
that is relevant to the proof of our index theorem.

Suppose $F$ is a closed, $G$-invariant subset of $N$.
Then there are two cohomology spaces associated to $F$ defined in 
\cite{PV1}: $\mathcal{H}^\infty(\mathfrak{g},N,N\setminus F)$, the relative equivariant cohomology of $N$, 
 and $\mathcal{H}^\infty_F(\mathfrak{g},N)$, the equivariant cohomology 
with compact support in $N$.

Representatives of cohomology classes in the former are pairs $(\eta, \xi)$, 
where $\eta\in\mathcal{A}^\infty(\mathfrak{g},N)$ and
$\xi\in\mathcal{A}^\infty(\mathfrak{g},N\setminus F)$, that are closed 
under the relative equivariant differential $D_{rel}(\eta,\xi) = 
(D\eta,\eta|_{N\setminus F}-D\xi)$, while cohomology classes in the latter 
are defined as follows:

Let $U\subset N$ be any open, $G$-invariant subset containing $F$.  We may 
consider the spaces $\mathcal{A}^{\infty}_U(\mathfrak{g}, N)$ of 
equivariant differential forms with support contained in $U$, and their 
corresponding cohomology spaces 
$\mathcal{H}^{\infty}_U(\mathfrak{g},N)$.

If we have two open subsets $V$ and $U$ with $F\subset V\subset U$, the 
inclusion $\mathcal{A}^{\infty}_V(\mathfrak{g},N)\hookrightarrow 
\mathcal{A}^{\infty}_U(\mathfrak{g},N)$ induces a map
\begin{equation*}
f_{U,V}:\mathcal{H}^{\infty}_V(\mathfrak{g},N)\rightarrow 
\mathcal{H}^{\infty}_U(\mathfrak{g},N),
\end{equation*}
and so we obtain the inverse system 
$(\mathcal{H}^{\infty}_U(\mathfrak{g},N), f_{U,V}, U,V\in\mathcal{F}_F)$,
where $\mathcal{F}_F$ is the family of all open, $G$-invariant 
neighbourhoods of $F$, 
letting us define the space of cohomology with support contained 
in $F$ as 
the inverse limit of this system.

By \cite{PV2}, all of the above can be extended to equivariant cohomology 
with generalized coefficients, including the morphism
\begin{equation}
p_F: 
\mathcal{H}^{\pm\infty}(\mathfrak{g},N,N\setminus F)\rightarrow 
\mathcal{H}^{\pm\infty}_F(\mathfrak{g},N)\label{pF}
\end{equation} defined in \cite{PV1} as follows:

Let $U$ be any open, $G$-invariant neighbourhood of $F$, and choose a 
cutoff function $\chi\in C^\infty(N)^G$ with support contained in $U$, 
such that $\chi\equiv 1$ on a smaller neighbourhood of $F$.  Let 
$(\eta,\xi)$ represent a class in 
$\mathcal{H}^{\pm\infty}(\mathfrak{g},N,N\setminus F)$ (so that 
$\eta|_{N\setminus F} = D\xi$), and set 
\begin{equation} p^\chi(\eta,\xi) = \chi\eta+d\chi\xi.
\end{equation}

By Proposition 3.14 in \cite{PV1}, $p^\chi(\eta,\xi)$ is an 
equivariantly closed form with support in $U$, whose class 
$p_U(\eta,\xi)$ in 
$\mathcal{H}^{\pm\infty}_U(\mathfrak{g},N)$ does not depend on $\chi$.
Moreover, $f_{U,V}\circ p_V = p_U$, so that we may define $p_F(\eta,\xi)$ 
to be the element defined by taking the inverse limit over invariant 
neighbourhoods of $F$.

\begin{remark}
An element of 
$\mathcal{H}^{\pm\infty}_F(\mathfrak{g},N)$ in the image of $p_F$ may be 
represented in computations by one of the forms $p^\chi(\eta,\xi)$.
\end{remark}

If $F$ is compact, there is a natural map 
\begin{equation}
\mathcal{H}^{\pm\infty}_F(\mathfrak{g},N)\rightarrow 
\mathcal{H}^{\pm\infty}_c(\mathfrak{g},N)\label{cpct}.
\end{equation}
The composition of $p_F$ with (\ref{cpct}) defines a map denoted $p_c$ in 
\cite{PV1}.

In the case where $N$ is a $G$-equivariant vector bundle we introduce two 
other complexes of differential forms: the 
complexes $\mathcal{A}^{\pm\infty}_{rdm}(\mathfrak{g},N)$ of differential 
forms that are rapidly decreasing in mean:

\begin{definition}
Suppose $N\rightarrow B$ is a $G$-equivariant vector bundle over the 
compact base $B$, and suppose $\beta:\mathfrak{g}\rightarrow N$ is an 
equivariant differential form on $N$ (possibly with generalized 
coefficients).  We say that $\beta$ is {\bf rapidly decreasing in mean} 
if for any smooth, compactly supported density $\rho(X)$ on 
$\mathfrak{g}$, the differential form $\beta_\rho = \int_\mathfrak{g} 
\beta(X)\rho(X)d X$ and all its derivatives are rapidly decreasing along 
the fibres of $N\rightarrow B$.
\end{definition}

The equivariant differential $D$ is well-defined on 
$\mathcal{A}^{-\infty}_{rdm}(\mathfrak{g},N)$, and so we may define the 
cohomology space $\mathcal{H}^{-\infty}_{rdm}(\mathfrak{g},N)$.

Note that we have the inclusions
\begin{equation}
\mathcal{A}^{\infty}_{rdm}(\mathfrak{g},N)\hookrightarrow 
\mathcal{A}^{-\infty}_{rdm}(\mathfrak{g},N)\hookleftarrow
\mathcal{A}^{-\infty}_c(\mathfrak{g},N).\label{incl}
\end{equation}

\subsection{Chern characters}
Suppose $\E=\E^+\oplus \E^-$ is a $\mathbb{Z}_2$-graded $G$-equivariant 
vector bundle over $N$, and let $\mathbb{A}$ be a $G$-invariant superconnection on $\E$, 
in the sense of Quillen \cite{Q}; see also \cite{MQ} or \cite{BGV}.  Thus 
$\mathbb{A}$ is an odd invariant operator on $\mathcal{A}(N,\E)$ which preserves the 
$\mathbb{Z}_2$-grading of $\E$, and satisfies the derivation property
\begin{equation*}
\mathbb{A}(\omega s) = d\omega s + (-1)^{\deg\omega}\omega\mathbb{A}s,
\end{equation*}
for any form $\omega$ on $N$, and any section $s$ of $\E$.
One example is the superconnection $\mathbb{A} = \nabla + L$ on 
$\E$ considered in \cite{MQ}, where
$\nabla$ is a $G$-invariant connection on $\E$ in the usual sense, and $L$ is an odd 
$G$-invariant endomorphism of $\E$.

\begin{definition}[\cite{BGV, PV}]
Given a $G$-invariant superconnection $\mathbb{A}$ on $\E$, we define the {\bf moment} 
of 
$\mathbb{A}$ to be the map $\mu^\mathbb{A} :\mathfrak{g}\rightarrow 
\mathcal{A}(N,\End(\E))^+$ given by
\begin{equation*}
\mu^\mathbb{A}(X) = \mathcal{L}(X) - [\iota_{X_N}, \mathbb{A}],
\end{equation*}
where $\mathcal{L}(X) = [d, \iota_{X_N}]$ is the Lie derivative in the 
direction of $X_N$.  

In the case $\mathbb{A} = \nabla + L$ mentioned 
above, the moment of $\mathbb{A}$ becomes simply $\mu(X) = \mathcal{L}(X) 
- \nabla_X$.

The {\bf equivariant curvature} of $\mathbb{A}$ is the map 
$\mathbb{F}(\mathbb{A}): \mathfrak{g}\rightarrow 
\mathcal{A}(N,\End(\E))^+$ given 
by 
$\mathbb{F}(\mathbb{A})(X) = \mathbb{A}^2 + \mu^\mathbb{A}(X)$, and the 
{\bf 
equivariant Chern character} of $(\E, \mathbb{A})$ is the equivariant 
differential form $\Ch(\mathbb{A},X) = 
\Str(e^{\mathbb{F}(\mathbb{A})(X)})$.
\end{definition}

The equivariant Chern character is equivariantly closed, so that 
$\Ch(\mathbb{A},X)$ defines a class in 
$\mathcal{H}^{\infty}(\mathfrak{g},N)$ equal to the Chern character of 
$\E$.

Now, if we are given a smooth, $G$-equivariant morphism $\sigma: 
\E^+\rightarrow \E^-$, define $\sigma^*$ using an invariant Hermitian 
metric on $\E$.  Then the map
\begin{equation*}
v_\sigma = \begin{pmatrix}0&\sigma^*\\\sigma&0\end{pmatrix}
\end{equation*}
defines an odd Hermitian endomorphism of $\E$, and
we can associate to it a differential form given by 
\begin{equation}
\Ch(\mathbb{A}^\sigma,X) = \Ch(\mathbb{A}(\sigma,1),X) = 
\Str(e^{\mathbb{F}(\mathbb{A}, \sigma, 
1)(X)}),\label{MQ}
\end{equation}
where $\mathbb{A}(\sigma,t) = \mathbb{A}+itv_\sigma,$
and $\mathbb{F}(\mathbb{A}, \sigma, t)(X)$ is the equivariant curvature 
of $\mathbb{A}(\sigma, t)$. Explicitly, we have
\begin{equation}
\mathbb{F}(\mathbb{A}, \sigma, t)(X) = -t^2v^2_\sigma +it[\mathbb{A}, 
v_\sigma] + 
\mathbb{F}(\mathbb{A})(X).\label{F}
\end{equation}
We remark that in the non-equivariant setting, the Chern character 
(\ref{MQ}) is essentially the form considered by Mathai-Quillen 
\cite{MQ}, in the case where $N$ is a vector bundle.   
We will denote by 
$\Ch_{MQ}(\sigma,X)$ the corresponding equivariant Chern character studied 
in \cite{PV1}.

If we define as well the transgression form $\eta(\mathbb{A},\sigma, t) = 
-i\Str(v_\sigma e^{\mathbb{F}(\mathbb{A},\sigma,t)})$, then on 
$N\setminus\supp(\sigma)$, we have the well-defined 
equivariant differential form 
$\beta(\mathbb{A}, \sigma)\in\mathcal{A}^\infty(\mathfrak{g}, 
N\setminus\supp(\sigma))$ given by
\begin{equation*}
\beta(\mathbb{A}, \sigma) = \int^\infty_0 \eta(\mathbb{A}, \sigma, t)d 
t.
\end{equation*}
Then $\Ch(\mathbb{A})\vert_{N\setminus\supp(\sigma)} = 
D\beta(\mathbb{A}, \sigma)$ \cite{PV1}, so that
\begin{equation} 
(\Ch(\mathbb{A}),\beta(\mathbb{A}, \sigma))\in  
\mathcal{H}^\infty(\mathfrak{g}, N, N\setminus\supp(\sigma)).\label{rel}
\end{equation}
This is the {\em relative Chern character} $\Ch_{rel}(\sigma,X)$ of 
\cite{PV1}.
With $F=\supp(\sigma)$, we can use the map (\ref{pF}) to obtain a 
class
\begin{equation*}
\Ch_{sup}(\sigma,X)=c_F(\Ch_{rel}(\sigma,X))\in
\mathcal{H}^\infty_{\supp(\sigma)}(\mathfrak{g}, N),
\end{equation*}
which is independent of the superconnection $\mathbb{A}$, and can be 
represented in computations by an equivariant form
\begin{equation*}
c(\sigma,\mathbb{A},\chi) = \chi\Ch(\mathbb{A}) + d\chi\beta(\mathbb{A}, 
\sigma),
\end{equation*}
where $\chi\in C^\infty(N)$ is a $G$-invariant cutoff function equal to 1 
on a neighbourhood of $\supp(\sigma)$, with support contained in $U$, for 
some $G$-invariant neighbourhood of $\supp(\sigma)$.

\begin{remark}
If $\sigma$ is elliptic, so that $\supp(\sigma)$ is compact, then we 
obtain a class 
$\Ch_c(\sigma)\in\mathcal{H}^\infty_c(\mathfrak{g},N)$ under 
the the natural map (\ref{cpct}) given in the previous subsection.
\end{remark}

Furthermore, we have the following theorem \cite{PV2}:
\begin{theorem}\label{tpv1}
Suppose $N$ is $G$-equivariant vector bundle over a manifold $B$.
Then if $\sigma$ satisfies suitable growth conditions along the 
fibres of $N$ (see \cite{PV2}), the form $\Ch_{MQ}(\sigma,X)$ is an element 
of 
$\mathcal{A}^\infty_{rdm}(\mathfrak{g},N)$ and represents the image of the 
class 
$\Ch_{sup}(\sigma,X)\in\mathcal{H}^\infty_{\supp(\sigma)}(\mathfrak{g},N)$ 
in $\mathcal{H}^\infty_{rdm}(\mathfrak{g},N)$.

Moreover, if the fibres of $\pi:N\rightarrow B$ are oriented, and the 
action of $G$ preserves the orientation, then we have 
$\pi_*\Ch_{MQ}(\sigma,X) = \pi_*\Ch_{sup}(\sigma,X)$ in 
$\mathcal{H}^\infty(\mathfrak{g},B)$. 
\end{theorem}

We now move from forms with smooth coefficients to those with generalized 
coefficients, which will allow us to shrink the support of our Chern 
character by using a $G$-invariant 1-form to modify the superconnection.

In \cite{PV2} we see that the above results carry over to equivariant 
cohomology with generalized coefficients.

Let $\lambda\in\mathcal{A}^1(\mathfrak{g},N)$ be a $G$-invariant 1-form.
We use $\lambda$ to deform the part of 
our superconnection of exterior degree one, obtaining a new 
superconnection 
$\mathbb{A}^{\sigma,\lambda} = \mathbb{A}(\sigma, \lambda, 1)$, according 
to
\begin{equation*}
\mathbb{A}(\sigma, \lambda, t) = \mathbb{A} +it(\sigma 
+\lambda),\:\mbox{ for }\: t\in\R \:\mbox{ and }\: v_\sigma + \lambda = 
\begin{pmatrix} \lambda & \sigma \\ \sigma^\ast & \lambda \end{pmatrix}.
\end{equation*}

As before, we set $\mathbb{F}(\mathbb{A}, \sigma, \lambda, t) = 
(\mathbb{A} + 
it(v_\sigma + \lambda))^2 + \mu^\mathbb{A}$, so that 
$\mathbb{F}(\mathbb{A}, 
\sigma, \lambda, 1)$ is the equivariant curvature of $\mathbb{A}^{\sigma, 
\lambda}$, $\Ch(\mathbb{A}^{\sigma,\lambda}) = 
\Str(e^{\mathbb{F}(\mathbb{A},\sigma,\lambda,1)})$ is the associated 
character 
form, and 
\begin{equation*}\eta(\sigma, \lambda, \mathbb{A}, t) = -i\Str((v_\sigma + 
\lambda)e^{\mathbb{F}(\mathbb{A}, \sigma, \lambda, t)})
\end{equation*}
the transgression form.

Then we may define $\beta(\mathbb{A}, \sigma, \lambda) = 
\int^\infty_0 \eta(\mathbb{A}, \sigma, \lambda, t)d t$, which is now 
well-defined on $N\setminus(\supp(\sigma)\cap C_\lambda)$, but only as a 
differential form with generalized coefficients \cite{PV2}.

We thus obtain a class
\begin{equation*}
\Ch_{rel}(\sigma, \lambda) = (\Ch(\mathbb{A}),\beta(\mathbb{A}, \sigma, 
\lambda)) \in 
\mathcal{H}^{-\infty}(\mathfrak{g}, N, N\setminus(\supp(\sigma)\cap 
C_\lambda)),
\end{equation*}
giving us
\begin{equation}
\Ch_{sup}(\sigma, \lambda) = 
c_F(\Ch_{rel}(\sigma, \lambda))\in 
\mathcal{H}^{-\infty}(\mathfrak{g},N,N\setminus F),
\end{equation}
where $F = \supp(\sigma)\cap C_\lambda$.  The class 
$\Ch_{sup}(\sigma,\lambda)$ is independent of $\mathbb{A}$, and can be 
represented by a differential form
\begin{equation*}
c(\sigma, \lambda, \mathbb{A},\chi) = \chi\Ch(\mathbb{A}) + 
d\chi\beta(\sigma, \lambda, \mathbb{A}),
\end{equation*}
where $\chi\in C^\infty(N)^G$ is equal to 1 on a neighbourhood of 
$\supp(\sigma)\cap C_\lambda$, and has support contained in a 
$G$-invariant neighbourhood $U$ of $\supp(\sigma)\cap C_\lambda$.

As in the smooth case, we have the following \cite{PV2}:
\begin{theorem}\label{tpv2}
Suppose that $N$ is a $G$-equivariant vector bundle over a $G$-manifold 
$B$.  If $\sigma$ and $\lambda$ satisfy suitable growth conditions along 
the fibres of $N\rightarrow B$, then 
$\Ch(\mathbb{A}^{\sigma,\lambda})\in\mathcal{A}^\infty_{rdm}(\mathfrak{g},N)$ 
and 
$\beta(\mathbb{A},\sigma,\lambda)\in 
\mathcal{A}^{-\infty}_{rdm}(\mathfrak{g},N\setminus (\supp(\sigma\cap 
C_\lambda))$, and $\Ch(\mathbb{A}^{\sigma,\lambda})$ represents the image 
of 
$\Ch_{sup}(\sigma,\lambda)$ in 
$\mathcal{H}^{-\infty}_{rdm}(\mathfrak{g},N)$.

Moreover, if the fibres of $\pi:N\rightarrow B$ are oriented, and the 
action of $G$ preserves the orientation, then the morphism 
$\pi_*:\mathcal{H}^{-\infty}_{rdm}(\mathfrak{g},N)\rightarrow 
\mathcal{H}^{-\infty}(\mathfrak{g},B)$ is well-defined, and 
$\pi_*\Ch(\mathbb{A}^{\sigma,\lambda}) = \pi_*\Ch_{sup}(\sigma,\lambda)$ 
in 
$\mathcal{H}^{-\infty}(\mathfrak{g},B)$.
\end{theorem}

In the case of a trivial bundle, we can define a class $P_\kappa 
\in\mathcal{H}^{-\infty}_{C_\kappa}(\mathfrak{g}, N)$ via $P_\kappa = 
c_F(2\pi, \beta(\kappa))$, where $F=C_\kappa$ and
 $\beta(\kappa) = -i\kappa\int^\infty_0 
e^{itD\kappa}d t = -2\pi i\kappa\delta_+(D\kappa)$, so that $D\beta = 
2\pi$ by (\ref{d2}).  If $U$ is a $G$-invariant neighbourhood of 
$C_\kappa$ and $\chi\in 
C^{\infty}(U)^G$ is equal to 1 on a neighbourhood of $C_\kappa$, then 
$P_\kappa$ can be represented by the form $P_U(\kappa,\chi) = 2\pi\chi 
+d\chi\beta(\kappa)$.

\subsection{The case of a contact manifold}
We return now to the case of a compact, co-oriented contact manifold 
$(M,E)$.  Consider the complex vector bundle 
$p:E^{0,1}\rightarrow M$ obtained from a $G$-invariant almost-CR structure on $M$.  
Equip $E^{0,1}$ with a $G$-invariant Hermitian metric $h$ compatible with 
the symplectic structure on $E$ and the 
almost-CR structure.  Let $\nabla$ be a $G$-invariant Hermitian 
connection on $E^{0,1}$ 
and let $F(X)$ be its equivariant curvature.

The symbol $\sigma_b$ (\ref{sig}) on $\pi_M^*\E$ is just the pullback by 
$q:T^*M\rightarrow E^*$ of the equivariant morphism 
$\sigma^{\C}_{E^{0,1}}:p^*(\bigwedge E^{0,even})\rightarrow p^*(\bigwedge 
E^{0,odd})$ defined in \cite{PV1}.  Furthermore, we have $\sigma^*_b = 
\sigma_b$ with respect to the metric $h$, so that $v^2_{\sigma_b} = 
\sigma_b^2\Id$, giving $\Ch_{MQ}(\sigma_b,X)$ ``Gaussian shape'' along the 
fibres of $E^*$ as in \cite{MQ}.

If we define the equivariant Todd form of $(E^{0,1},\nabla)$ for 
$X\in\mathfrak{g}$ sufficiently small by
\begin{equation*}
\Td(E^{0,1},X) = 
\det\nolimits_\C\left(\frac{F(X)}{e^{F(X)}-1}\right),
\end{equation*}
then we have \cite{PV1}
\begin{equation*}
\Ch_{MQ}(\sigma^\C_{E^{0,1}},X) = (2\pi i)^n 
p^*(\Td(E^{0,1},X)^{-1})\Th_{MQ}(E^{0,1}) \text{ in } 
\mathcal{H}^\infty_{rdm}(\mathfrak{g},E^{0,1}),\label{thom}
\end{equation*}
where $\Th_{MQ}(E^{0,1})$ is an equivariant version of the Thom form 
defined in \cite{MQ}, and $n$ is the complex rank of $E^{0,1}$.

If we pull back the above result to $T^*M$, then we obtain
\begin{equation}
\Ch_{MQ}(\sigma_b,X) = (2\pi i)^n 
\pi^*_M(\Td(E^{0,1},X)^{-1})q^*(\Th_{MQ}(E^{0,1})),\label{tdth}
\end{equation}
where $\Ch_{MQ}(\sigma_b,X) = \Ch(\mathbb{A}^\sigma_b,X)$.

Let $\theta$ be the canonical 1-form on $T^*M$.
Since the action of $G$ on $M$ is assumed to be elliptic, we know that 
$F=\supp(\sigma_b)\cap C_\theta = T^*_GM\cap E^0 = \{0\}$ is compact.
Thus $\Ch_{sup}(\sigma_b,\theta)$ defines a class in 
$\mathcal{H}^{-\infty}_c(\mathfrak{g},T^*M)$ under the mapping 
(\ref{cpct}).

Denote by $\Ch_{BV}(\sigma_b,X)$ the Chern character of \cite{BV1}, given 
by
\begin{equation*}
\Ch_{BV}(\sigma_b,X) = \Str(e^{F(\mathbb{A},\sigma_b,\theta,1)(X)}),
\end{equation*}
which is an element of $\mathcal{A}^\infty_{rdm}(\mathfrak{g},T^*M)$.

By Theorem \ref{tpv2}, the images of $\Ch_{BV}(\sigma_b,X)$ and
$\Ch_{sup}(\sigma_b,\theta)$ under the maps induced 
by the 
inclusions (\ref{incl}) coincide in 
$\mathcal{H}^{-\infty}_{rdm}(\mathfrak{g},T^*M)$.

Simlarly, by Theorem \ref{tpv1}, a representative 
of $\Ch_{sup}(\sigma_b)$ in $\mathcal{H}^\infty_{rdm}(\mathfrak{g},T^*M)$, is given by $\Ch_{MQ}(\sigma_b,X)$, 
which we relate to $\Ch_{BV}(\sigma_b,X)$ using following two lemmas from 
\cite{PV2}:

\begin{lemma}\label{l3}
Let $\kappa$ be a $G$-invariant 1-form on $N$, and define $P_\kappa$ as 
above.  Then:
\begin{enumerate}
\item Under the natural map 
$\mathcal{H}^{-\infty}_{C_\kappa}(\mathfrak{g},N) \rightarrow 
\mathcal{H}^{-\infty}(\mathfrak{g}, N)$, the image of $P_\kappa$ is equal 
to 1.
\item $\Ch(\sigma, \kappa) = P_\kappa\wedge\Ch(\sigma)$ in 
$\mathcal{H}^{-\infty}_{\supp(\sigma)\cap C_\kappa}(\mathfrak{g},N)$.
\end{enumerate}
\end{lemma}

\begin{lemma}
\label{l6}
If $\theta\vert_{\supp(\sigma)} = \lambda\vert_{\supp(\sigma)}$, then 
$\supp(\sigma)\cap C_\theta = \supp(\sigma)\cap C_\lambda = T^\ast_G M\cap 
E^0$, and $\Ch(\sigma, \theta) = \Ch(\sigma, \lambda)$ in 
$\mathcal{H}^{-\infty}_{T^\ast_G M\cap E^0}(\mathfrak{g},T^\ast M)$.
\end{lemma}

Together, the two above lemmas give:
\begin{prop}
Let $i:E^0 \hookrightarrow 
T^\ast M$ be the inclusion of $E^0$, 
and define $\lambda = i^\ast\theta$.  Using the splitting $T^\ast 
M = E^\ast\oplus E^0$, consider $\lambda$ as a form on all of 
$T^\ast M$, by taking $\lambda\vert_{E^\ast} = 0$, and $\lambda\vert_{E^0} 
= i^\ast\theta$.

Then, $\lambda$ and $\theta$ agree on $\supp(\sigma_b) = E^0$, and we have 
\begin{equation}
\Ch_{BV}(\sigma_b,X) = P_\lambda(X)\wedge \Ch_{MQ}(\sigma_b,X)
\text{ in } 
\mathcal{H}^{-\infty}_{rdm}(\mathfrak{g},T^*M).\label{eq11}
\end{equation}
\end{prop}

\section{Calculation of the index}
We now apply the results of the previous section, in the case of a 
compact, co-oriented contact manifold $(M,E)$ to the Berline-Vergne index 
formula (\ref{bv}).

Recall that the equivariant $\hat{A}$-class is defined for any real 
$G$-equivariant vector bundle $\E\rightarrow M$, with 
$G$-equivariant connection $\nabla$ and corresponding equivariant 
curvature $F(X)$ by
\begin{equation*}
\hat{A}(\E,X) = 
\det\nolimits^{1/2}_\R\left(\frac{F(X)}{e^{F(X)/2}-e^{-F(X)/2}}\right),
\end{equation*}
with the choice of square root depending on orientation.  The equivariant
$\hat{A}$-class of $TM\rightarrow M$ is denoted by $\hat{A}(M,X)$.

The form $D_\R(N(g),X)$ associated to the 
normal bundle is defined in \cite{BV1} as follows:

\begin{definition}
For $g\in G$, let $F_N(X)$, $X\in\mathfrak{g}(g)$, denote the 
equivariant curvature of $N(g)$ with respect to a $G(g)$-equivariant 
connection.  Then $D_\R(N(g),X)$ is the $G(g)$-equivariantly closed from 
on $M(g)$ given for $X\in \mathfrak{g}(g)$ by
\begin{equation*}
D_\R(N(g),X) = \det\nolimits_\R(1-g^N e^{F_N(X)}),
\end{equation*}
where $g^N$ denotes the lifted action of $g\in G$ 
on $N(g)$.
\end{definition}

The splitting $TM = E\oplus \R$ given by the choice of contact form, together with the invaraince of the Reeb field imply that the normal bundle $N(g)$ is contained within $E|_{M(g)}$, and thus inherits a complex structure from $E$.

Hence we may similarly define the form $D_\C(N(g),X)$ using the complex determinant in place of the real determinant used above.
Note that using the complex structure on $N(g)$, we may write $N(g)\otimes\C 
= N(g)\oplus \overline{N(g)}$ and obtain:
\begin{equation*}
D_\R(N(g),X) = D_\C(N(g)\otimes\C,X) = D_\C(N(g),X)D_\C(\overline{N(g)},X).
\end{equation*}

We are now ready to state the main theorem of this article:
\begin{theorem}
Let $(M,E)$ be a compact, co-oriented contact manifold of dimension 
$2n+1$, 
and let $G$ be a 
compact Lie group acting elliptically on $M$.  Let $g\in G$, and let $k(g)$ be the locally constant function defined by
$\dim M(g) = 2k(g)+1$.\label{main}

The $G$-equivariant index of $\dirac$ is the generalized 
function on $G$ whose germ at $g\in G$, is given,
for $X\in\mathfrak{g}(g)$ sufficiently small, by
\begin{equation}
\ind^G(\dirac)(ge^X) = \int\limits_{M(g)}(2\pi i)^{-k(g)} \label{fpf}
\frac{\Td(E(g),X)\mathcal{J}(E(g),X)}{D_\C(N(g),X)}.
\end{equation}
\end{theorem}

\subsection{The formula near the identity}
We first consider the index formula for group elements $e^X$, for 
$X\in\mathfrak{g}$ sufficiently small.  The calculation in this case is 
simpler, and employs the results of \cite{PV2} directly.  The general 
result will then follow an analogous approach.

\begin{theorem}\label{t1}
For $X\in\mathfrak{g}$ sufficiently small, we have
\begin{equation}
\ind^G(\dirac)(e^X) = \frac{1}{(2\pi i)^n}\int\limits_M 
\Td(E,X)\mathcal{J}(M,X).
\end{equation}
\end{theorem}

\begin{proof}
The formula of Berline-Vergne for the equivariant index of a 
transversally elliptic operator is given by
\begin{equation}
\ind^G(\dirac)(e^X) = \frac{1}{(2\pi i)^{(2n+1)}}\int\limits_{T^\ast 
M}\pi^*_M(\hat{A}^2(M,X))\Ch_{BV}(\sigma_b,X).\label{BV}
\end{equation}

Using the splitting $TM = E\oplus\R$, and the almost-complex structure on 
$E$, we have that 
\begin{equation}\label{bert}
\hat{A}^2(M,X) = 
\hat{A}^2(E,X) = \Td(E\otimes\C,X) = \Td(E^{1,0},X)\Td(E^{0,1}, X).
\end{equation}

By (\ref{tdth}) and (\ref{eq11}), we have 
\begin{align}\label{ernie}
\Ch_{BV}(\sigma_b,X) &= P_\lambda(X)\Ch_{MQ}(\sigma_b,X)\\\nonumber
& = (2\pi i)^n 
P_\lambda(X) 
\pi^*_M(\Td(E^{0,1},X)^{-1})q^*(\Th_{MQ}(E^{0,1})).
\end{align}

Combining (\ref{bert}) and (\ref{ernie}), we find

\begin{multline}\label{garrr}
\pi^*_M(\hat{A}^2(M,X))\Ch_{BV}(\sigma_b,X) =\\ (2\pi i)^n 
\pi^*_M(\Td(E^{1,0},X)^{-1})q^*(\Th_{MQ}(E^{0,1}))P_\lambda(X)
\end{multline}
in $\mathcal{H}^{-\infty}_{rdm}(\mathfrak{g},T^*M)$.

\begin{lemma}\label{l101}
In terms of the projections $q:T^*M\rightarrow E^*$ and $p:E^*\rightarrow 
M$ we have
\begin{equation*}
q_*P_\lambda(X) = 2\pi i\, p^*\mathcal{J}(E,X).
\end{equation*}
That is, $(\pi_M)_*P_\lambda = p_*q_*P_\lambda = 2\pi 
i\mathcal{J}(E,X)$ in $\mathcal{H}^{-\infty}(\mathfrak{g},M)$.
\end{lemma}

\begin{proof}
A representative of $P_\lambda$ is given by 
\begin{equation*}
P^\chi_\lambda = 2\pi\chi + d\chi\wedge \beta(\lambda) = 2\pi\chi -2\pi i 
d\chi\wedge\lambda\delta_+(D\lambda),
\end{equation*}
where $\chi$ is a cutoff function with support in a 
neighbourhood of $E^*$, and $\chi\equiv 1$, on $E^*$.

Since $\lambda$ is a form on $E^0$, and $\chi$ is constant on 
$E^*$, $P_\lambda$ is independent of $E^*$, 
and so it remains to calculate the integral over the fibre of $E^0 = 
M\times\R$.  Let $t$ be the co-ordinate along the fibre, and write 
$\chi=\chi(t)$.  Then $\chi(t)$ 
is supported on a neighbourhood of $t=0$, with $\chi(0)=1$,
and $\lambda$ may be written as 
$\lambda = -t\alpha$, for $\alpha$ a contact form on $M$. Thus $D\lambda = 
D(-t\alpha) = \alpha\wedge dt - tD\alpha$, and $P_\lambda$ becomes
\begin{align*}
P_\lambda & = 2\pi\chi(t) -2\pi i\chi^\prime(t)\,dt\wedge(-t\alpha) 
\delta_+(\alpha\wedge dt - tD\alpha)\\
&=2\pi\chi(t) -2\pi i\alpha\wedge t\chi^\prime(t)\,dt \delta_+(-tD\alpha).
\end{align*}

Thus, the integral over $\R$ becomes, with the help of the identities in 
Section 5.1, 
\begin{align*}
\int^\infty_{-\infty}P_\lambda &=-2\pi i\alpha\int^\infty_{-\infty}\chi^\prime(t)t\delta_+(-tD\alpha)\,d t\\
&= -2\pi i\alpha\left[ \int^\infty_0 \chi^\prime(t)\delta_{-}(D\alpha)\,d t - 
\int^0_{-\infty} \chi^\prime(t)\delta_{+}(D\alpha)\,d t\right]\\
&= -2\pi i\alpha\left[ -\delta_{-}(D\alpha) - \delta_{+}(D\alpha)\right]\\
&=2\pi i\alpha\delta_0(D\alpha),
\end{align*}
and we obtain our result.
\end{proof}

Let $\Td(E,X)$ denote the cohomology class of the Todd form 
$\Td(E^{1,0},X)$, and write $\Th_{MQ}(E^{0,1}) = \Th_{MQ}(E^*)$ using 
the isomorphism (\ref{psi}). 
Then using Lemma \ref{l101} and (\ref{garrr}) in the index formula 
(\ref{BV}), we obtain
\begin{align*}
\ind^G(\dirac)(e^X) &= (2\pi i)^{-(2n+1)}\int\limits_{T^*M} (2\pi i)^n 
\pi^*_M(\Td(E,X))q^*\Th_{MQ}(E^*)P_\lambda(X)\\
&= \frac{1}{(2\pi 
i)^n}\int\limits_{E^*}p^*(\Td(E,X)\mathcal{J}(E,X))\Th_{MQ}(E^*)\\
&= \frac{1}{(2\pi i)^n}\int\limits_M \Td(E,X)\mathcal{J}(E,X).\qedhere
\end{align*}
\end{proof}

\subsection{Fixed Point Formula}
The general calculation of the push-forward of (\ref{bv}) from $T^*M(g)$ to 
$M(g)$ is analogous to the proof given above, since by Proposition 
\ref{prop1}, we know that $(M(g),E(g))$ is again a contact manifold on 
which $G(g)$ acts elliptically.  The primary added difficulty comes from 
the appearance of the action of $g\in G$ in the Chern character form of 
\cite{BV1}.

\vspace{8pt}

\begin{proof}[Proof of Theorem \ref{main}:] 
By Proposition 
\ref{prop1}, $(M(g),E(g))$ is again a contact manifold, and we have the 
splitting
\begin{equation*}
T^*M(g) = E^*(g)\oplus \R\alpha^g.
\end{equation*}

Let $\dim M(g) = 2k(g)+1$, so that $E(g)$ is a vector bundle over $M(g)$ 
of complex rank $k(g)$.

Denote by $j:T^*M(g)\rightarrow T^*M$ the inclusion of the $g$-fixed point 
set in $T^*M$.  Let $H=G(g)$, with Lie algebra $\mathfrak{h}$.  Let 
$\mathbb{A}^g = j^*\mathbb{A}$, $\theta^g = j^*\theta$ and $\sigma_b^g = 
j^*\sigma_b$ denote the restrictions of the superconnection, canonical 
1-form and symbol of the previous section to $T^*M(g)$.
Let $p,q$ denote the projections $p:E^*(g)\rightarrow 
M(g)$ and $q:T^*M(g)\rightarrow E^*(g)$.

The Chern character $\Ch_{BV}^g(\sigma_b)(X)$ of 
Berline and Vergne is given 
by
\begin{equation*}
\Ch_{BV}^g(\sigma_b)(X) = \Str(g^\E\cdot 
j^*(e^{\mathbb{F}(\mathbb{A},\sigma_b,\theta)(X)})).
\end{equation*}
By \cite{BV1}, since $\sigma_b$ is $G$-transversally elliptic, 
$\sigma_b^g$ is 
$H$-transversally elliptic, and $\Ch_{BV}^g(\sigma_b)$ defines a class in 
$\mathcal{H}^\infty_{rdm}(\mathfrak{h},T^*M(g))$.

Let $j^*\nabla = \nabla_E\oplus\nabla_N$ denote the decomposition of the 
restriction of the invariant Hermitian connection $\nabla$ on $E^{0,1}$ into 
connections on $E^{0,1}(g)$ and $N(g)$, respectively.  Since $\mathbb{A} = 
q^*\nabla$, we have $\mathbb{A}^g = \mathbb{A}_E\oplus\nabla_N$, where 
$\mathbb{A}_E = q^*\nabla_E$.  

By Lemma 19 of \cite{BV1}, the canonical 1-form on $T^*M(g)$ is 
simply the restriction $\theta^g$ of the canonical 1-form $\theta$ on 
$T^*M$.   Since $\theta^g$ is invariant under the action of $g$, we have 
\begin{equation*}
\Ch_{BV}^g(\sigma_b,X) = e^{iD\theta^g}\Str(g^\E\cdot 
e^{\mathbb{F}(\mathbb{A}_E\oplus\nabla_N,\sigma_b)(X)}).
\end{equation*}
\begin{lemma}
Let $V$ be a complex vector space of dimension $k$, and let a Lie group 
$G$ act on $V$, such that the action commutes with the natural $U(k)$ 
action on $V$.  Let $\rho: U(k)\rightarrow \bigwedge V^*$ denote the 
representation of $U(k)$ on $\bigwedge V^*$ as in \cite{MQ}, and let 
$w\in\mathfrak{u}(k)$ be a skew-symmetric Hermitian matrix.  Then for any 
$g\in G$, we have
\begin{equation*}
\Str(g\cdot e^{\rho(w)}) = \det\nolimits_\C(1-g\cdot e^{-w}).
\end{equation*}
\end{lemma}

\begin{proof}
Since the action of $G$ on $\bigwedge V^*$ commutes with the 
representation $\rho$, the actions of $g$ and $w$ can be simultaneously 
diagonalized.
\end{proof}

\vspace{8pt}

Using the above Lemma we may write
\begin{equation}\label{mq2}
\Ch^g_{BV}(\sigma_b,X) = 
e^{iD\theta^g}\Ch(\mathbb{A}_E,\sigma_b^g,X)D_\C(\overline{N}(g),X),
\end{equation}
using
\begin{equation*}
\det\nolimits_\C(1-g^\E\cdot(j^*e^{-\mathbb{F}(\mathbb{A}_E\oplus\nabla)(X)})) 
= 
\det\nolimits_\C(1-e^{-\mathbb{F}(\mathbb{A}_E)(X)})\det\nolimits_\C(1-g^N 
\cdot e^{-F(\nabla_N)(X)}),
\end{equation*}
since $g$ acts trivially on $T^*M(g)$.

The form $\Ch(\mathbb{A}_E,\sigma_b^g,X)$ appearing in (\ref{mq2}) is 
simply 
the 
Mathai-Quillen form $\Ch_{MQ}(\sigma_b^g,X)$ on $E^{0,1}(g)$.  We again 
use Theorem \ref{tpv2} and Lemmas \ref{l3} and \ref{l6} as follows:

The class $e^{iD\theta^g(X)}\Ch_{MQ}(\sigma_b^g,X) = 
\Ch(\mathbb{A}_E,\sigma_b^g,\theta^g,X)\in 
\mathcal{A}^\infty_{rdm}(\mathfrak{h},T^*M(g))$ represents
$\Ch_{sup}(\sigma_b^g,\theta^g,X)\in\mathcal{H}^{-\infty}_{supp(\sigma_b^g)\cap 
C_{\theta^g}}(\mathfrak{h},T^*M(g))$ in
$\mathcal{H}^{-\infty}_{rdm}(\mathfrak{h},T^*M(g))$,
and we have
\begin{equation*}
\Ch_{sup}(\sigma^g,\theta^g,X) = \Ch_{sup}(\sigma^g,\lambda^g,X) = 
P_{\lambda^g}(X)\Ch_{sup}(\sigma^g,X).
\end{equation*}
Since a representative of $\Ch_{sup}(\sigma_b^g,X)$ is 
$\Ch_{MQ}(\sigma_b^g,X)$,
we have
\begin{equation}\label{ch}
\Ch_{BV}(\sigma_b,X)=P_{\lambda^g}(X)\Ch_{MQ}(\sigma_b^g,X)D_\C(\overline{N}(g),X)
\end{equation}
in $\mathcal{H}^{-\infty}_{supp(\sigma_b^g)\cap 
C_{\theta^g}}(\mathfrak{h},T^*M(g))$.

As before, $\Ch_{MQ}^g(\sigma_b)$ is the pull-back of a form on $E^*(g)$, 
while $P_{\lambda^g}$ depends only on $\R\alpha^g$.
Integrating $P_{\lambda^g}$ over $\R$ proceeds the same as in the proof of 
Theorem \ref{t1}, giving $2\pi i\mathcal{J}(E(g),X)$ as the 
result.

Finally, we substitute (\ref{ch}) into (\ref{bv}) and use (\ref{tdth}) and Lemma \ref{l101} to obtain (\ref{fpf}).
\end{proof}

Suppose now that $(B,\omega,\Phi)$ is a Hamiltonian $G$-manifold, which is pre-quantizable in the sense of \cite{GGK}.
Let $\mathbb{L}\rightarrow B$ be a $G$-equivariant pre-quantum line bundle, and let $\pi:M\rightarrow B$ be the unit circle bundle inside of $\mathbb{L}$ with respect to a Hermitian metric.  A {\em $G$-equivariant pre-quantization} of $(B, \omega, \Phi)$ is defined in \cite{GGK} to be the pair $(M,\tilde{\alpha})$, where $\tilde{\alpha}$ is a connection form on $M$, such that $iD\tilde{\alpha}(X) = \pi^*(\omega-\Phi(X))$.\footnote{We are using the convention here that $M$ is a principal $U(1)$-bundle, and that $\mathfrak{u}(1) = i\R$.}

Let $\dob$ denote the Dolbeault-Dirac operator on sections of $\bigwedge T^{0,1}B$, and let $\sigma_m = \sigma(\dob)\otimes\Id_{\mathbb{L}^{\otimes m}}$. 

Since the form $\alpha = i\tilde{\alpha}$ is a contact form on $M$, and the action of $G\times U(1)$ on $M$ is elliptic with respect to $E=\ker(\alpha)$, our index formula in this case provides the following:

\begin{corollary}\label{pb}
We have the following equality of generalized functions on $G\times U(1)$:
\begin{equation}
\ind^{G\times U(1)}(\dirac)(g,u) = \label{corr} 
\sum\limits_{m\in \Z} u^{-m}\ind^G(\sigma_{m})(g).
\end{equation}
\end{corollary}

\begin{proof}
With the right identifications, this result can be viewed as a special case of Th\'eor\`eme 25 in \cite{BV2} for $H=U(1)$, and the details of the proof are similar. 

We need to check that, for any fixed $(g,u)\in G\times U(1)$, the formula holds in a sufficiently small neighbourhood of $(g,u)$ in $G(g)\times U(1)$.  That is, for $X\in\mathfrak{g}(g)$ and $\phi\in\R$ sufficiently small, we need to show that
\begin{equation}
\ind^{G\times U(1)}(\dirac)(ge^X,ue^{i\phi})= \label{goal} \sum\limits_{m\in\Z}u^{-m}e^{-im\phi}\ind^G(\sigma_{m})(ge^X).
\end{equation}
For any $v\in U(1)$, we have $M(g,v) = \{y\in M|g\cdot y = y\cdot v\}$.  When $M(g,v)$ is non-empty, $U(1)$ acts freely on $M(g,v)$, and we denote $B(g)^v = M(g,v)/U(1)$. The fixed-point set $B(g)$ is a (finite) disjoint union of the spaces $B(g)^v$.

Since $\mathbb{L} \cong M\times_{U(1)}\C$, the action of $g\in G$ on the fibres of $\mathbb{L}|_{B(g)^v}$ is scalar multiplication by $v\in U(1)$. Thus, $\Ch_g(\mathbb{L}^{\otimes m},X)|_{B(g)^v} = v^me^{im\omega(X)}$, and we have
\begin{equation*}
\ind^G(\sigma_{m})(ge^X) = \sum\limits_{\substack{v\in U(1)\\M(g,v)\neq \emptyset}}\:\int\limits_{B(g)^v}(2\pi i)^{-k(g)}\frac{\Td(B(g)^v,X)}{D_\C(N_B(g),X)}v^me^{im\omega(X)}.
\end{equation*}
Thus, the only contribution to the right-hand side of (\ref{goal}) comes from $B(g)^u$ (provided $M(g,u)$ is non-empty), in which case we can apply the Poisson summation formula to obtain
\begin{multline*}
\sum\limits_{m\in\Z}u^{-m}e^{-im\phi}\ind^G(\sigma_m)(ge^X) =\\ \int\limits_{B(g)^u}(2\pi i)^{-k(g)}\frac{\Td(B(g)^u,X)\delta_0(\omega^g(X)-\phi)}{D_\C(N_B(g),X)}.
\end{multline*}

Using the index formula (\ref{fpf}), the left-hand side of (\ref{goal}) is given by
\begin{equation*}
\int\limits_{M(g,u)}(2\pi i)^{-k(g)}\frac{\Td(E(g,u),(X,\phi))\,\alpha^{g,u}\delta_0(D\alpha^{g,u}(X)-\phi)}{D_\C(N_M(g,u),(X,\phi))}.
\end{equation*}
The pre-quantization condition implies that $D\alpha^{g,u}(X,i\phi) = \pi^*\omega^g(X) - \phi$, and since the forms $\Td(E(g,u))$ and $D_\C(N_M(g,u))$ are the pullback to $M(g,u)$ of the corresponding forms on $B(g)^u$, the result follows.
\end{proof}

\subsection*{Acknowledgements}
Many of the ideas in this paper, including the statements of the main 
theorems and an outline of their proof, were 
first suggested by my thesis advisor, Eckhard Meinrenken.  I am grateful 
to him for suggesting this problem, and for much helpful advice and 
feedback on this paper.

I would also like to thank Paul-Emile Paradan for his helpful discussion 
while at Utrecht in the summer of 2007.  
The special case of Corollary \ref{pb} above was suggested by both Meinrenken 
and Paradan.

\bibliographystyle{plain}
\bibliography{index}

\end{document}